\documentclass[12pt,reqno]{amsart}
\usepackage{amsmath}
\usepackage{amssymb}
\usepackage{amstext}
\usepackage{a4wide}
\usepackage{graphicx}
\allowdisplaybreaks \numberwithin{equation}{section}
\usepackage{color}
\usepackage{cases}

\numberwithin{equation}{section}

\newtheorem{theorem}{Theorem}[section]

\newtheorem{lemma}[theorem]{Lemma}

\theoremstyle{definition}

\newtheorem{definition}[theorem]{Definition}

\theoremstyle{remark}
\newtheorem{remark}[theorem]{Remark}

\begin{document}

\title[travelling circular vortex pair for gSQG]
{Existence and stability of smooth travelling circular pairs for the generalized surface quasi-geostrophic equation}

 \author{Daomin Cao, Guolin Qin,  Weicheng Zhan, Changjun Zou}

\address{Institute of Applied Mathematics, Chinese Academy of Sciences, Beijing 100190, and University of Chinese Academy of Sciences, Beijing 100049,  P.R. China}
\email{dmcao@amt.ac.cn}
\address{Institute of Applied Mathematics, Chinese Academy of Sciences, Beijing 100190, and University of Chinese Academy of Sciences, Beijing 100049,  P.R. China}
\email{qinguolin18@mails.ucas.edu.cn}
\address{Institute of Applied Mathematics, Chinese Academy of Sciences, Beijing 100190, and University of Chinese Academy of Sciences, Beijing 100049,  P.R. China}
\email{zhanweicheng16@mails.ucas.ac.cn}
\address{Institute of Applied Mathematics, Chinese Academy of Sciences, Beijing 100190, and University of Chinese Academy of Sciences, Beijing 100049,  P.R. China}
\email{zouchangjun17@mails.ucas.ac.cn}

%\thanks{This work is partially supported by ARC}

\begin{abstract}
     In this paper, we construct smooth travelling counter-rotating vortex pairs with circular supports for the generalized surface quasi-geostrophic equation. These vortex pairs are analogues of the Lamb dipoles for the two-dimensional incompressible Euler equation. The solutions are
     obtained by maximization of the energy over some appropriate classes of admissible functions. We establish the uniqueness of maximizers and compactness of maximizing sequences in our variational setting. Using these facts, we further prove the orbital stability of the circular vortex pairs for the gSQG equation.

\end{abstract}

\maketitle

\section{Introduction}

In this paper, we consider the following generalized surface quasi-geostrophic (gSQG) equation
\begin{align}\label{1-1}
	\begin{cases}
		\partial_t\theta+\mathbf{u}\cdot \nabla \theta =0&\text{in}\ \mathbb{R}^2\times (0,T)\\
		\ \mathbf{u}=\nabla^\perp\psi,\ \ \psi=(-\Delta)^{-s}\theta     &\text{in}\ \mathbb{R}^2\times (0,T),\\
	\end{cases}
\end{align}
where $0<s<1$, $(a_1,a_2)^\perp=(a_2,-a_1)$, $\theta$ is the potential temperature being transported by the velocity field $\mathbf{u}$ generated by $\theta$ and $\psi$ is the stream function. The operator $(-\Delta)^{-s}$ in $\mathbb{R}^N$ is the standard inverse of the fractional Laplacian and is given by
\begin{equation}\label{1-2'}
	(-\Delta)^{-s}\theta(x)=\int_{\mathbb{R}^N}G_s(x-y)\theta(y)dy,\ \ G_s(z)=\frac{c_{N,s}}{|z|^{N-2s}},
\end{equation}
where
\begin{equation}\label{1-3'}
	c_{N,s}=\pi^{-\frac{N}{2}}{2^{-2s}}\frac{\Gamma(\frac{N-2s}{2})}{\Gamma(s)},
\end{equation}
and $\Gamma$ is the Euler gamma function. In physics, the gSQG equation is closely related to the temperature propagation on a surface and aggregation-diffusion processes.

Let us remark that the cases $s=1$ and $s={1}/{2}$ correspond to the Euler equation in vorticity formulation and the inviscid surface quasi-geostrophic (SQG) equation, respectively. In the last century, the two-dimensional Euler equation has been intensively studied and the global well-posedness of initial data in $L^1\cap L^\infty$ was proved by Yudovich \cite{Yud}. As the mathematical analogy with the classical three-dimensional incompressible Euler equation, the inviscid SQG equation is another type of active scalar equation, which describes the atmosphere circulation near the tropopause and to track the ocean dynamics in the upper layers \cite{Lap}. In recent years, thanks to the extension method of $(-\Delta)^{-s}$ with $0<s<1$ developed by Caffarelli and Silvestre \cite{Cafs}, there has been tremendous interest in the study of the gSQG equation. However, the Cauchy problem for gSQG equation is extremely delicate. It is difficult to extend the Yudovich theory of weak solutions to the gSQG equation because the velocity is, in general, not Lipschitz continuous. So far we only know the local well-posedness of sufficiently regular initial data \cite{CCG16, CCG20, Chae, Con, Kis1}, and finite time singularities were constructed for a large class of initial data \cite{Kis1,Kis2}. .

Global weak solutions for the SQG equation were first constructed in $L^2$ by Resnick \cite{R95}. Later,  Marchand \cite{Mar08} extended Resnick’s result to the class of initial data belonging to $L^p$ with $p>{4}/{3}$.  Non-uniqueness of weak solutions was recently shown in \cite{BSV19}. A special case of weak solutions are the patch type solutions. For the existences, regularity and uniqueness of patch solutions to the gSQG equation, please see, e.g., \cite{CCG16, CCG16-2, CCG18, CG19, CGS20, HH15, HM17} and the references therein.

Global existence of smooth solutions for the gSQG equation is also a matter of concern. It is well known that all radially symmetric functions $\theta$ are stationary solutions to the gSQG equation due to the structure of the nonlinear term. Exhibiting other global smooth solutions seems a challenging issue. Castro, C\'ordoba and G\'omez-Serrano \cite{ CCG20} constructed the first family of  non-trivial global smooth solutions for the SQG equation by using bifurcation theory. Later, Gravejat and Smets \cite{ GS17} constructed smooth travelling solutions for the SQG eqaution via variational methods. Godard-Cadillac \cite{Go0} generalized Gravejat and Smets's result to the gSQG equation. Very recently, Ao et al. \cite{Asqg} successfully applied the Lyapunov-Schmidt reduction method to construct travelling and rotating solutions (with concentrated vorticity) to the gSQG equation.

In this paper, we consider travelling solutions for the gSQG equation, which means solutions  having the fixed form and travelling with a constant velocity. Up to a rotation, we may assume, without loss of generality, that the travelling solution $\theta(x,t)$ satisfies
\begin{equation}\label{1-4}
	\theta(x,t)=\omega(x-Wt\mathbf{e}_1), \quad \forall\, t\in\mathbb{R},
\end{equation}
where $\omega(x)$ is some profile function defined on $\mathbb{R}^2$, $\mathbf{e}_1=(1,0)$ and $W\geq 0$ is the travelling speed.  At this stage, we assume both $\psi$ and $\theta$ are sufficiently regular. Set $\Psi=(-\Delta)^{-s}\omega$. For travelling solutions, we can use \eqref{1-4} to rewrite the first equation in \eqref{1-1} as
\begin{equation*}
	(\nabla^\perp\Psi-W\mathbf{e}_1)\cdot\nabla\omega=0,
\end{equation*}
which is equivalent to
\begin{equation}\label{1-5}
	\nabla^\perp(\Psi-Wx_2)\cdot\nabla\omega=0.
\end{equation}
As remarked by Arnol'd \cite{Arn}, a natural way of obtaining solutions to the stationary problem \eqref{1-5} is to impose that $\Psi-Wx_2$ and $\omega$ are locally functional dependent.

In 1906, Lamb \cite{Lamb} noted an explicit travelling solution for the two-dimensional Euler equation (i.e. $s=1$), which is now generally referred to as the Lamb dipole or Chaplygin-Lamb dipole; see \cite{Mel}. Let $\Psi_C$ be the function defined by
\begin{equation*}
\Psi_C(x)=\begin{cases}Wx_2-\frac{2WJ_1(\lambda^{1/2}r)}{\lambda^{1/2}J_1'(c_0)r}x_2,\ \ \ \  &\text{for}\ \ 0<r:=\sqrt{x_1^2+x_2^2}\le c_0\lambda^{-1/2},\\
{c_0^2x_2}/{\lambda r^2},\ \ \ \ &\text{for}\ \ r>c_0\lambda^{-1/2},\end{cases}
\end{equation*}
where $J_1$ is the first-order Bessel function of the first kind, $c_0=3.8317\cdots$ is the first positive zero of $J_1$, and $\lambda>0$ is a vortex strength parameter. Then $\Psi_C$ is a solution of
\begin{equation*}
\begin{cases}-\Delta \Psi=\lambda(\Psi-Wx_2)_+ \ \  \text{in}\ \Pi:=\{x\in \mathbb{R}^2\mid x_2>0\},&\\
\Psi\to 0\  \text{as}\  r\to \infty,\ \ \ \Psi=0\ \text{on}\ \partial \Pi,&\\
\Psi(x_1,x_2)=-\Psi(x_1,-x_2),\ \ \forall\,x\in \mathbb{R}^2.&\\
\end{cases}
\end{equation*}
where $f_+$ denotes the positive part of $f$. The Chaplygin-Lamb dipole $\omega_C$ has the following form
\begin{equation*}
  \omega_C(x_1,x_2)=-\omega_C(x_1,-x_2)=\lambda(\Psi_C(x)-Wx_2)_+,\ \ \forall\,x\in\Pi.
\end{equation*}
Observe that $\Psi_C-Wx_2$ is the stream function of a flow whose vorticity is $\omega_C$. This flow is symmetric in the $x_1$-axis and contains a pair of vortices bounded by a single circle. Moreover, its far-field approaches a uniform flow in the $x_1$-direction. Exact solutions of \eqref{1-5} with $s=1$ are known only in special cases. Besides those exact solutions, the existence (and abundance) of travelling vortex pairs for the two-dimensional Euler equation has been rigorously established; see \cite{AC19, B88, CLZ, T} and the references therein. We also remark that there is an analogue of the Chaplygin-Lamb dipole for the three-dimensional axis-symmetric Euler equation, which was introduced by Hill in 1894; see, e.g., \cite{AF86, Choi20, Hill}.

An interesting question is that are there some analogues of the Chaplygin-Lamb dipole for the generalized surface quasi-geostrophic equation? We mention that although many travelling solutions for the gSQG equation have been established (see, e.g., \cite{Asqg, CQZZ, Go0, GS17}), it is still not known whether some circular vortex pairs exist. The purpose of this paper is to provide the first construction of such travelling solutions.
\begin{theorem}\label{Ee}
	For $0<s<1$, there exists a smooth function $\omega_L$ such that
	\begin{equation}\label{1-8}
		\begin{cases}
			\omega_L(x_1,x_2)=-\omega_L(x_1,-x_2),\ \ x\in\mathbb{R}^2,\\
			supp(\omega_L)=\overline{B_1(0)}\,\,\text{is the unit disk in the plane},\\
			\theta(x,t)=\omega_L(x-Wt\mathbf{e}_1)\,\, \text{is a solution of }\,\,\eqref{1-1},\ \text{for some constant} \,\,W>0,
		\end{cases}
	\end{equation}
where $supp(\cdot)$ denotes the support of a function.
\end{theorem}

The solution $\omega_L$ is of the form
\begin{equation*}
  \omega_L=\lambda(\Psi_L-Wx_2)_+\ \ \text{in}\  \Pi,
\end{equation*}
where $\lambda$ is a positive number and $\Psi_L:=(-\Delta)^{-s}\omega_L$. Moreover, $\Psi_L$ is a solution of
\begin{equation*}
\begin{cases}(-\Delta)^s \Psi=\lambda(\Psi-Wx_2)_+ \ \ \ \ &\text{in}\ \Pi,\\
\Psi\to 0\  \text{as}\  r\to \infty,\ \ \ \Psi=0\ \ &\text{on}\ \partial \Pi,\\
\Psi(x_1,x_2)=-\Psi(x_1,-x_2),\ \ &\forall\,x\in \mathbb{R}^2.\\
\end{cases}
\end{equation*}
It is easy to see that $\omega_L$ is actually an analogue of the Chaplygin-Lamb dipole for the two-dimensional Euler equation and Hill's vortex for the three-dimensional axis-symmetric Euler equation.

We construct the solutions by using the variational method. Since the desired flows are symmetric about the $x_1$-axis, we can restrict our attention henceforth to the upper half-plane $\Pi$. Let $\bar{x}=(-x_1, x_2)$ be the reflection of $x$ in the $x_1$-axis. Denote
\begin{equation}\label{1-9}
	G_\Pi(x,y)=c_{2,s}\left(\frac{1}{|x-y|^{2-2s}}-\frac{1}{|x-\bar{y}|^{2-2s}}\right), \quad\forall\, x, y\in \Pi,
\end{equation}
and
\begin{equation}\label{1-10}
	\mathcal{G}_s\omega(x)=\int_\Pi G_\Pi(x,y)\omega(y) dy, \quad \forall\, x\in \Pi.
\end{equation}	
We introduce the kinetic energy of the fluid
\begin{equation*}
  E(\omega)=\frac{1}{2}\int_\Pi \omega(x)\mathcal{G}_s \omega(x) dx,
\end{equation*}
and its impulse
\begin{equation*}
  I(\omega)=\int_{\mathbb{R}^2}x_2\omega(x)dx.
\end{equation*}
For $1/2<s<1$ and $0<\mu,\nu,\lambda<\infty$, we set the space of admissible functions
\begin{equation*}
  \mathcal{A}^s_{\mu, \nu}:=\Big{\{}\omega\in L^2(\Pi)\mid \omega\geq 0, \int_\Pi x_2\omega dx=\mu, \int_\Pi\omega dx\leq \nu\Big{\}}.
\end{equation*}
We define the energy functional $\mathcal{E}_\lambda$ by
\begin{equation*}
  \mathcal{E}_\lambda(\omega)=E(\omega)-\frac{1}{2\lambda}\int_\Pi \omega^2 dx,\ \ \ \ \omega\in \mathcal{A}^s_{\mu, \nu}.
\end{equation*}
We will consider the maximization of the energy functional $\mathcal{E}_\lambda$ relative to $\mathcal{A}^s_{\mu, \nu}$. When $0<s\le1/2$, the kinetic energy is not well-defined in $L^1(\Pi)\cap L^2(\Pi)$ due to the singularity of $G_\Pi$. For this reason, we
define the space of admissible functions in this case by
\begin{equation*}
  \mathcal{A}^s_{\mu, \nu}:=\Big{\{}\omega\in L^\infty(\Pi)\mid \omega\geq 0, \int_\Pi x_2\omega dx=\mu, \int_\Pi\omega dx\leq \nu\Big{\}}.
\end{equation*}
It seems rather difficult to show the existence of maximizers for $\mathcal{E}_\lambda$ over $\mathcal{A}^s_{\mu, \nu}$ directly because of the lack of compactness of maximizing sequences. To overcome this difficulty, we add an additional constraint on the $L^\infty$ norm and consider the maximization problem in the space
\begin{equation*}
  \mathcal{A}^{s,\Gamma}_{\mu, \nu}:=\Big{\{}\omega\in L^\infty(\Pi)\mid 0\leq \omega\leq \Gamma, \int_\Pi x_2\omega dx=\mu, \int_\Pi\omega dx\leq \nu\Big{\}},
\end{equation*}
for some $\Gamma>0$. We prove the existence of maximizers for $\mathcal{E}_\lambda$ over $\mathcal{A}^{s,\Gamma}_{\mu, \nu}$, and fortunately these maximizers are uniformly bounded by a constant that is independent of $\Gamma$. Hence, the constraint $\omega\leq \Gamma$ is indeed redundant for $\Gamma$ large. In other words, we obtain the existence of maximizers for $\mathcal{E}_\lambda$ over $\mathcal{A}^s_{\mu, \nu}$.

We will prove that there exists a constant $\mu_0>0$ such that if $0<\mu\nu^{-1}\lambda^{\frac{1}{2s}}\leq \mu_0$, then each maximizer of $\mathcal{E}_\lambda$ relative to $\mathcal{A}^s_{\mu, \nu}$ is indeed a desired solution after some suitable $x_1$-translation. By using the method of moving planes, we can also prove the uniqueness of maximizers whenever $0<\mu\nu^{-1}\lambda^{\frac{1}{2s}}\leq \mu_0$ in the sense that any two maximizers differ by only a translation in the $x_1$-direction. Denote
\begin{equation}\label{1-11}
	S^s_{\mu, \nu,\lambda}=\sup_{\omega\in \mathcal{A}^s_{\mu, \nu}}\mathcal{E}_\lambda(\omega),\quad \Sigma^s_{\mu, \nu,\lambda}=\big{\{}\omega\in \mathcal{A}^s_{\mu, \nu}\mid \mathcal{E}_\lambda(\omega)=S^s_{\mu, \nu,\lambda}\big{\}}.
\end{equation}
\begin{theorem}\label{Uq}
	Let $0<s<1$. Suppose $0<\mu\nu^{-1}\lambda^{\frac{1}{2s}}\leq \mu_0$, then there exists an $\omega^s_{\mu, \nu,\lambda}\in \mathcal{A}^s_{\mu, \nu}$ such that
\begin{equation*}
  \Sigma^s_{\mu, \nu,\lambda}=\{\omega^s_{\mu, \nu,\lambda}(\cdot+c\mathbf{e}_1)\mid c\in\mathbb{R}\}.
\end{equation*}
Moreover, $\text{supp}(\omega^s_{\mu, \nu,\lambda})$ is a half disk centered at the origin.
\end{theorem}

In \cite{AF86}, Amick and Fraenkel established the uniqueness of Hill's vortex. The uniqueness of the Chaplygin-Lamb dipole was shown by Burton \cite{B96}. These uniqueness results, in conjunction with characterization of the energy of solutions, are fundamental in establishing compactness of maximizing sequences and stability of solutions. We refer the interested readers to \cite{AC19, B05, B21, BNL13, Choi20} and the references therein concerning the stability of vortex solutions for the Euler equation.

As for the gSQG equation, there seem to be very few results on stability. However, since the solution we construct here has a special energy characteristic, we can prove its orbital stability. Similar to Burton \cite{BNL13}, we introduce the following $L^p$-regular solutions.

\begin{definition}\label{def2}
	$\xi\in L^\infty_{loc}(\left[0,T\right), L^1(\mathbb{R}^2))\cap  L^\infty_{loc}(\left[0,T\right), L^p(\mathbb{R}^2))$ is called a \emph{ $L^p$ regular solution} of \eqref{1-1} if $\xi$ satisfies \eqref{1-1}
	in the sense of distributions, such that $E(\xi(t,\cdot))$, $I(\xi(t,\cdot))$ and $||\xi(t,\cdot)||_q$ for $1\leq q\leq p$ are constant for $t\in[0, T)$. Moreover, if $\xi_0$ is odd symmetric in $x_2$, then  $\xi(t, \cdot)$ is also odd symmetric in $x_2$.	
\end{definition}

Generally speaking, a $L^p$-regular solution is weak solution to \eqref{1-1} such that its kinetic energy, impulse and $L^q$ norms for $1\leq q\leq p$ conserve; see \cite{BSV19} for some discussion about the conservation laws.

We set $p_s=\infty$ if $0<s\leq {1}/{2}$ and $p_s=2$ if ${1}/{2}<s<1$. For a function $\xi$ defined on $\Pi$, we denote its odd extension to the whole plane by $\bar{\xi}$. Our stability theorem is as follows.

\begin{theorem}\label{Os}
	The circular vortex-pair $\omega_L$ obtained in Theorem \ref{Ee} is orbital stable in the sense that, for arbitrary $\nu>0$, $M>0$ and $\varepsilon>0$, there exists $\delta>0$ such that for non-negative function $\xi_0\in L^1\cap L^{p_s}(\Pi)$ with $||\xi_0||_1\leq \nu$, $||\xi_0||_{p_s} <M$ and
	\begin{equation}\label{1-14}
		\inf_{c\in \mathbb{R}}\Big{\{}\|\xi_0-\omega_L(\cdot+c\mathbf{e}_1)\|_2+\|x_2(\xi_0-\omega_L(\cdot+c\mathbf{e}_1))\|_1\Big{\}} \leq \delta,
	\end{equation}
	if there exists a $L^{p_s}$-regular solution $\xi(t)$ with initial data $\overline{\xi_0}$ for $t\in [0, T)$ with $0<T\leq \infty$, then
	\begin{equation}\label{1-15}
		\inf_{c\in \mathbb{R}} \Big{\{}\|\xi(t)-\omega_L(\cdot+c\mathbf{e}_1)\|_2+\|x_2(\xi(t)-\omega_L(\cdot+c\mathbf{e}_1))\|_1\Big{\}}\leq \varepsilon, \quad \forall\, t\in[0,T).
	\end{equation}
\end{theorem}

\begin{remark}\label{rem2}
Orbital stability of the Chaplygin-Lamb dipole was considered in \cite{AC19}(see \cite{Choi20} for Hill's vortex). Our stability result is similar to that of \cite{AC19}. These results reveal some similarities between the Euler equation and the gSQG equation.
\end{remark}

This paper is organized as follows. In section 2, we give some basic estimates and show the existence of maximizers of $\mathcal{E}_\lambda$ relative to $\mathcal{A}^s_{\mu,\nu}$. In section 3, we establish the uniqueness of maximizers. We prove the compactness of maximizing sequences in section 4 by using the concentrated compactness method of Lions \cite{L84}. Finally, we show the orbital stability in section 5.

In what follows, the symbol $C$ denotes a general positive constant that may change from line to line. For $p\in [1,+\infty]$, we use $p'$ to denote its conjugate exponent, that is, $1/p'+1/p=1$. We denote by $B_R(x)$ the open ball in $\mathbb{R}^2$ of center $x$ and radius $R>0$. If $\Omega\subset\mathbb{R}^2$ is measurable, then $\text{meas}\,({\Omega})$ denotes the two-dimensional Lebesgue measure of $\Omega$. Let ${1}_\Omega$ denote the characteristic function of $\Omega\subset\mathbb{R}^2$.

\section{Existence of maximizers}
In this section, we show the existence of maximizers. We first give some basic estimates used frequently later.

 \begin{lemma}\label{lm1}
 Suppose that $0<s<1$ and ${1}/{s}<p\leq \infty$. Then there is a positive constant $C$, depending only on $s$ and $p$, such that if $0\le \omega\in L^1(\Pi)\cap L^p(\Pi)$, then
 	\begin{equation}\label{2-1}
 		\|\mathcal{G}_s\omega\|_\infty\leq C\left(\|\omega\|_1+\|\omega\|_p\right),
 	\end{equation}
 	\begin{equation}\label{2-2}
 		E(\omega)\leq C\|\omega\|_1\left(\|\omega\|_1+\|\omega\|_p\right).
 	\end{equation}

 	%(ii) $\int_\Pi\int_\Pi \omega_1(x) G(x,y)\omega_2(y)dxdy\leq C||\omega_1||_1(||\omega_2||_1+||\omega_2||_2)$,\\ 		(iv) $|E(\omega_1)-E(\omega_2)|\leq C||\omega_1-\omega_2||_1(||\omega_1+\omega_2||_1+||\omega_1+\omega_2||_2)$,\\
 	
 \end{lemma}
\begin{proof}
	
    Since $\omega\geq 0$ and $p>{1}/{s}$, by H\"older inequality, we have
	\begin{align*}
		0\leq \mathcal{G}_s\omega(x)&\leq \int_\Pi \frac{c_{2,s}}{|x-y|^{2-2s}} \omega(y)dy\\
		&\leq \int_{|x-y|<1}\frac{c_{2,s}}{|x-y|^{2-2s}} \omega(y)dy+\int_{|x-y|\geq 1} \frac{c_{2,s}}{|x-y|^{2-2s}} \omega(y)dy \\
		&\leq c_{2,s}\left(  \int_{|x-y|<1}\frac{1}{|x-y|^{p'(2-2s)}}dy\right)^{\frac{1}{p'}}||\omega||_p+c_{2,s}||\omega||_1\\
		&\leq C(||\omega||_1+||\omega||_p).
	\end{align*}
 This proves \eqref{2-1}. Note that \eqref{2-2} follows from \eqref{2-1} easily by the definition of $E(\omega)$. The proof is thus completed.
\end{proof}

\begin{lemma}\label{lm2}
Suppose that $0<s<1$ and ${1}/{s}<p\leq \infty$. Then for every ${1}/{s}<q\leq p$, there is a positive constant $C$, depending only on $s$ and $q$, such that
if $0\le \omega\in L^1(\Pi)\cap L^p(\Pi)$, then
	\begin{equation}\label{2-3}
		\mathcal{G}_s\omega(x)\leq C\left(x_2^{2s-\frac{2}{q}}||\omega||_q+ x_2^{2s-3}||x_2\omega||_1\right),\ \ \forall\, x\in \Pi.
	\end{equation}
In addition,
	\begin{equation}\label{2-4}
		\mathcal{G}_s\omega(x)\rightarrow 0\, \ \ \text{as} \ \  |x|\rightarrow \infty.
	\end{equation}
\end{lemma}
\begin{proof}
	Notice that, by the mean value theorem, there holds
\begin{equation*}
  G_\Pi(x,y)\leq \frac{4x_2y_2}{|x-y|^{4-2s}}, \ \ x, y\in \Pi.
\end{equation*}
Therefore, we have
	\begin{align*}
		\mathcal{G}_s\omega(x)&=\int_\Pi G_\Pi(x,y)\omega(y) dy\\
		&\leq \int_{|x-y|<2x_2} \frac{c_{2,s}}{|x-y|^{2-2s}} \omega(y) dy + \int_{\Pi\cap\{|x-y|>2x_2\}} \frac{4x_2y_2}{|x-y|^{4-2s}} \omega(y) dy\\
		&\leq C\left(x_2^{2s-\frac{2}{q}}\|\omega\|_q+ x_2^{2s-3}\|x_2\omega\|_1\right),
	\end{align*}
	which proves \eqref{2-3}.
	
	For $|x|$ large,  by \eqref{2-1}, we derive
	\begin{align*}
		0\leq \mathcal{G}_s\omega(x)&=\int_\Pi G_\Pi(x,y) \omega(y)dy\\
		&\leq \int_{|y|<\frac{|x|}{2}}\frac{c_{2,s}}{|x-y|^{2-2s}} \omega(y)dy+\int_{\Pi} G_\Pi(x,y) \omega(y)1_{\Pi\setminus B_{\frac{|x|}{2}}(0)}dy \\
		&\leq C\left(\frac{1}{|x|^{2-2s}} \|\omega\|_1 +\|\omega1_{\Pi\setminus B_{\frac{|x|}{2}}(0)}\|_1+\|\omega1_{\Pi\setminus B_{\frac{|x|}{2}}(0)}\|_p\right)\\
		&=o(1),
	\end{align*}
	which is \eqref{2-4} and completes the proof of Lemma \ref{lm2}.
\end{proof}

Since the energy $\mathcal{E}$ is invariant under translations in $x_1$-direction, to control maximizers, we need to take the Steiner symmetrization in $x_1$-variable. For a given nonnegative function $\omega$, we shall say that $\omega$ is Steiner symmetric if
\begin{equation}\label{2-5}
	\omega(x_1,x_2)=\omega(-x_1,x_2) \ \ \ \text{and}\ \ \ \omega \,\,\text{is non-increasing in }\,x_1\, \text{for}\,x_1>0.	
\end{equation}
We have the following result, which can be found in \cite{AC19, B88, T}.
\begin{lemma}[Steiner symmetrization]\label{lm3}
Suppose that $0<s<1$ and ${1}/{s}<p\leq \infty$. For $\omega\geq0$ satisfying $\omega\in L^1(\Pi)\cap L^p(\Pi)$ and $x_2\omega \in L^1(\Pi)$, there exists $\omega^*\geq 0$ satisfying \eqref{2-5} and
	\begin{align}\label{2-6}
		\begin{cases}
		\|\omega^*\|_q=||\omega||_q,\ \ \forall\,1\leq q\leq p,\\
		\|x_2\omega^*\|_1=\|x_2\omega\|_1,\\
		E(\omega^*)\geq E(\omega).
		\end{cases}
	\end{align}
\end{lemma}

For a Steiner symmetric function, we have the following estimate:
\begin{lemma}\label{lm4}
Suppose that $0<s<1$ and ${1}/{s}<p\leq \infty$. Then there is a constant $C>0$, depending only on $s$ and $p$, such that
\begin{equation}\label{2-7}	
		\mathcal{G}_s\omega(x)\leq C\left(|x_1|^{-\frac{1}{2}}\|\omega\|_1+ |x_1|^{-\frac{1}{2p}}\|\omega\|_p+ \frac{x_2}{|x_1|^{2-s}}\|x_2\omega\|_1\right),\ \ x\in \Pi,
	\end{equation}
for all nonnegative $\omega\in L^1(\Pi)\cap L^p(\Pi)$ that are Steiner symmetric.
\end{lemma}

\begin{proof}
	For $x\in\Pi$ fixed, let
\begin{equation*}
	\omega_1(y)=\left\{
	\begin{array}{lll}
		\omega(y), \ \  & \text{if} \ \ |y_1-x_1|<\sqrt{|x_1|},\\
		0, & \text{if} \ \ |y_1-x_1|\geq\sqrt{|x_1|}.
	\end{array}
	\right.
\end{equation*}
	Using  equation (2.11) in \cite{B88}, it is easy to see that
\begin{equation*}
  \|\omega_1\|_p\leq \left(\frac{|x_1|^{\frac{1}{2}}}{|x_1|}\right)^{\frac{1}{p}}\|\omega\|_p=|x_1|^{-\frac{1}{2p}}\|\omega\|_p.
\end{equation*}
Hence, by \eqref{2-1}, we have
	\begin{equation}\label{2-8}
		\mathcal{G}_s\omega_1(x)\leq C\left(\|\omega_1\|_1+\|\omega_1\|_p\right)\leq C\left(|x_1|^{-\frac{1}{2}}\|\omega_1\|_1+ |x_1|^{-\frac{1}{2p}}\|\omega_1\|_p\right).
	\end{equation}
Letting $\omega_2=\omega-\omega_1$, we have
	\begin{align}\label{2-9}
		|\mathcal{G}_s\omega_2(x)|&= c_{2,s}\int_{|x-y|>\sqrt{|x_1|}} \left(\frac{1}{|x-y|^{2-2s}}-\frac{1}{|x-\bar{y}|^{2-2s}}\right)\omega(y)dy\nonumber\\
		& \leq C\int_{|x-y|>\sqrt{|x_1|}} \frac{x_2y_2}{|x-y|^{4-2s}}\omega(y)dy \\
		&\leq \frac{Cx_2}{|x_1|^{2-s}}\|x_2\omega\|_1,\nonumber
	\end{align}	
    which, together with \eqref{2-8}, gives \eqref{2-7}.

\end{proof}

Now, we prove the existence of maximizer for $\mathcal{E}_\lambda$ over $\mathcal{A}^{s}_{\mu, \nu}$. Set
\begin{equation}\label{2-10}
	\tilde{\omega}(x)=\frac{\lambda^{-\frac{1}{s}}}{\nu}\omega(\lambda^{-\frac{1}{2s}}x)
\end{equation}
 It is easy to see that if $\omega\in \mathcal{A}^s_{\mu,\nu}$, then $\tilde{\omega}(x)\in \mathcal{A}^s_{\mu\nu^{-1}\lambda^{{1}/{2s}},1}$ and $\mathcal{E}_1(\tilde{\omega})=\lambda^{1-\frac{1}{s}}\nu^{-2}\mathcal{E}_\lambda(\omega)$. Thus, without loss of generality, we may assume that $\lambda=\nu=1$.
For simplicity, in what follows, we denote $\mathcal{A}^s_\mu=\mathcal{A}^s_{\mu,1}$, $\Sigma^s_\mu=\Sigma^s_{\mu,1,1}$, $\mathcal{E}(\omega)=E(\omega)-\frac{1}{2}\int_\Pi \omega^2dx$ and  $S^s_{\mu}=\sup_{\omega\in \mathcal{A}^s_{\mu}}\mathcal{E}(\omega)$. In the following, we always assume that $\lambda=\nu=1$.

We consider the cases $0<s\leq {1}/{2}$ and  ${1}/{2}<s<1$ separately.
\subsection{$\mathbf{{1}/{2}<s<1}$} In this subsection, we consider the simpler case ${1}/{2}<s<1$. We need the following lemma concerning maximum value.
\begin{lemma}\label{lm5}
 If $0<\mu<\infty$, then $0<S^s_\mu<\infty$.
\end{lemma}
\begin{proof}
	By taking $p=2$ in \eqref{2-2} and Young's inequality, we have for $\omega\in L^1(\Pi)\cap L^2(\Pi)$,
\begin{equation*}
  \begin{split}
     \mathcal{E}(\omega) &=E(\omega)-\frac{1}{2}\int_\Pi \omega^2dx \\
       & \le C\|\omega\|_1\left(\|\omega\|_1+\|\omega\|_2\right)-\frac{1}{2}\int_\Pi \omega^2dx \\
       & \le C\|\omega\|_1^2.
  \end{split}
\end{equation*}
	Thus, $S^s_{\mu}<\infty$ due to $\|\omega\|_1\leq 1$ for all $\omega \in \mathcal{A}^s_\mu$. Let $\omega_1:=1_B$ with $B=\Pi\cap B_r(0)$ for some $r>0$  satisfying $\int_\Pi x_2 \omega_1dx=\mu$. Setting $\omega_\tau(x):=\tau^3\omega_1(\tau x)$ for $\tau>0$, then one can check that
	\begin{align*}
		&\int_\Pi x_2 \omega_\tau dx= \int_\Pi x_2 \omega_1 dx=\mu,\\
		&\int_\Pi \omega_\tau dx= \tau\int_\Pi \omega_1dx,\\
		&\mathcal{E}(\omega_\tau)=\tau^{4-2s}\left(E(\omega_1)-\frac{\tau^{2s}}{2}\int_\Pi \omega_1^2dx\right).
	\end{align*}
	Hence, we can choose $\tau$ sufficiently small such that $\omega_\tau\in \mathcal{A}^s_\mu$ and $\mathcal{E}(\omega_\tau)>0$, which achieves the proof.	
\end{proof}
    We are able to prove the existence of maximizers.
\begin{lemma}\label{lm6}
	For $\mu>0$, there exists $\omega_\mu^s \in \mathcal{A}^s_{\mu}$ such that
	$$\mathcal{E}(\omega^s_\mu)=\sup_{\omega\in \mathcal{A}^s_{\mu}} \mathcal{E}(\omega).$$	
\end{lemma}
\begin{proof}
Let $\{\omega_j\}_{j=1}^\infty\subset \mathcal{A}^s_\mu$ be a maximizing sequence. By the definition of $\mathcal{E}$ and \eqref{2-2}, we have $$\|\omega_j\|_2^2=2E(\omega_j)-2\mathcal{E}(\omega_j)\leq C\|\omega_j\|_1(||\omega_j\|_1+\|\omega_j\|_2)-\mathcal{E}(\omega_j),$$ which, by Young's inequality,  implies
	$$\|\omega_j\|_2^2\leq C-\mathcal{E}(\omega_j)\leq C-S^s_\mu\leq C.$$
	Thus, $\{\omega_j\}_{j=1}^\infty$ is uniformly bounded in $L^2(\Pi)$.
	
	In view of Lemma \ref{lm3}, we may assume that $\omega_j$ is Steiner symmetric by replacing $\omega_j$ with its Steiner symmetrization. Since $\{\omega_j\}_{j=1}^\infty$ is uniformly bounded in $L^2(\Pi)$, passing to a subsequence (still denoted by $\{\omega_j\}_{j=1}^\infty$), we may assume
	$\omega_j \rightarrow \omega$ weakly in $L^2(\Pi)$ as $j\to \infty$. It is easy to verify that $$\int_\Pi x_2\omega dx\leq \mu \ \ \text{and}\ \ \int_\Pi \omega dx\leq 1.$$ We are going to show the convergence of energy.
	
	Indeed, on the one hand, by Lemmas \ref{lm1} and \ref{lm4}, we have
	\begin{align*}
	 &2E(\omega_j)=\int_\Pi\int_\Pi \omega_j(x){G}_\Pi(x,y)\omega_j(y)dxdy\\
		&\leq \int_{|x_1|<R, 0<x_2<R} \int_{|y_1|<R, 0<y_2<R}\omega_j(x){G}_\Pi(x,y)\omega_j(y)dxdy\\
        &\ \ \ \ \ \ \ \ \ \ \ \ \ \ \  + 2\int_{x_2\geq R} \omega_j(x)\mathcal{G}_s\omega_j(x)dx+ 2\int_{|x_1|\geq R} \omega_j(x)\mathcal{G}_s\omega_j(x)dx\\
		&\leq  \int_{|x_1|<R, 0<x_2<R} \int_{|y_1|<R, 0<y_2<R}\omega_j(x){G}_\Pi(x,y)\omega_j(y)dxdy+2R^{-1}\|\mathcal{G}_s\omega_j\|_\infty\|x_2\omega_j\|_1\\
		&\ \ \ \ \ \ \ \ \ \ \ \ \ \ \    + C\left(R^{-\frac{1}{2}}\|\omega_j\|_1^2+R^{-\frac{1}{4}}\|\omega_j\|_2\|\omega_j\|_1+R^{s-2}\|x_2\omega_j\|_1^2\right)\\
        &\leq  \int_{|x_1|<R, 0<x_2<R} \int_{|y_1|<R, 0<y_2<R}\omega_j(x){G}_\Pi(x,y)\omega_j(y)dxdy+CR^{-1}\left(\|\omega_j\|_1+\|\omega_j\|_2\right)\|x_2\omega_j\|_1\\
        &\ \ \ \ \ \ \ \ \ \ \ \ \ \ \    + C\left(R^{-\frac{1}{2}}\|\omega_j\|_1^2+R^{-\frac{1}{4}}\|\omega_j\|_2\|\omega_j\|_1+R^{s-2}\|x_2\omega_j\|_1^2\right).\\
	\end{align*}
Observing that $G_\Pi(x,y)\in L_{loc}^{2}(\overline{\Pi}\times\overline{\Pi})$, we obtain
\begin{equation*}
  \limsup _{j\rightarrow \infty}E(\omega_j)\leq E(\omega)
\end{equation*}
by first letting $j\rightarrow \infty$ and then $R\rightarrow \infty$.
	
	On the other hand, one has
	$$	2E(\omega_j)=\int_\Pi \omega_j(x)\mathcal{G}_s\omega_j(x)dx\geq \int_{|x_1|<R, 0<x_2<R} \int_{|y_1|<R, 0<y_2<R}\omega_j(x)G_\Pi(x,y)\omega_j(y)dxdy,$$
	which implies
\begin{equation*}
  \liminf _{j\rightarrow \infty}E(\omega_j)\geq E(\omega)
\end{equation*}
by first letting $j\rightarrow \infty$ and then $R\rightarrow \infty$.

Therefore, we conclude that $$\lim _{j\rightarrow \infty }E(\omega_j)= E(\omega).$$ Hence, we have $$\mathcal{E}(\omega)= E(\omega)-\frac{1}{2}\int_\Pi \omega^2dx\geq \lim _{j\rightarrow \infty}E(\omega_j)-\liminf _{j\rightarrow \infty}\frac{1}{2}\int_\Pi \omega_j^2dx=S_\mu^s.$$

We now claim that $\int_\Pi x_2\omega dx=\mu$. Indeed, suppose not, then there exists some $\tau>0$ such that the function
\begin{equation*}
	\omega_\tau(x_1, x_2):=\left\{
	\begin{array}{lll}
		\omega(x_1,x_2-\tau), \ \  & \text{if} \ \ x_2>\tau,\\
		0, & \text{if} \ \ x_2\le\tau.
	\end{array}
	\right.
\end{equation*}
belongs to $\mathcal{A}_\mu^s$. A simple calculation then yields that
\begin{equation*}
  S^s_\mu=\mathcal{E}(\omega)<\mathcal{E}(\omega_\tau)\le S^s_\mu.
\end{equation*}
This is a contradiction. The proof is thus complete.	
\end{proof}

From the proof of Lemma \ref{lm6}, we can also obtain the following result.
\begin{lemma}\label{lmadd1}
If $0<\mu_1<\mu_2<\infty$, then $S^s_{\mu_1}<S^s_{\mu_2}$.
\end{lemma}

\subsection{$\mathbf{0<s\leq {1}/{2}}$} In this subsection, we consider all the remaining cases $0<s\le{1}/{2}$. As mentioned above, the energy functional $\mathcal{E}$ is not well-defined in $L^1(\Pi)\cap L^2(\Pi)$ when $0<s\le{1}/{2}$. One can not obtain the uniformly boundedness in $L^2$ of a maximizing sequence as in the proof of Lemma \ref{lm6}. To overcome this difficulty, we first consider the maximization problem in the space
\begin{equation*}
  \mathcal{A}^{s,\Gamma}_{\mu}:=\Big{\{}\omega\in L^\infty(\Pi)\mid 0\leq \omega\leq \Gamma, \int_\Pi x_2\omega dx=\mu, \int_\Pi\omega dx\leq \nu\Big{\}},
\end{equation*}
where $\Gamma>1$ is a parameter.

Denote $S^{s,\Gamma}_{\mu}=\sup_{\omega\in \mathcal{A}^{s,\Gamma}_{\mu,}}\mathcal{E}(\omega)$. Arguing as in the proof of Lemma \ref{lm5}, we have the following result.
\begin{lemma}\label{lm7}
 If $0<\mu<\infty$, then $0<S^{s,\Gamma}_{\mu}<\infty$.
\end{lemma}

We now prove the existence of maximizers of $\mathcal{E}$ relative to $\mathcal{A}^{s,\Gamma}_{\mu}$.
\begin{lemma}\label{lm8}
	For $\mu>0$ and $\Gamma>1$, there exists $\omega^{s,\Gamma}_\mu \in \mathcal{A}^{s,\Gamma}_{\mu}$ such that
	$$\mathcal{E}(\omega^{s,\Gamma}_\mu)=\sup_{\omega\in \mathcal{A}^{s,\Gamma}_{\mu}} \mathcal{E}(\omega).$$	
\end{lemma}
\begin{proof}
Let $\{\omega_j\}_{j=1}^\infty\subset \mathcal{A}^{s,\Gamma}_{\mu}$ be a maximizing sequence. By the definition of $\mathcal{A}^{s,\Gamma}_{\mu}$, we know that $\{\omega_j\}_{j=1}^\infty$ is uniformly bounded in $L^1(\Pi)\cap L^\infty(\Pi)$. In view of Lemma \ref{lm3}, we may assume that $\omega_j$ is Steiner symmetric by replacing $\omega_j$ with its Steiner symmetrization. Let $p\in (1/s, +\infty)$. Then, by passing to a subsequence (still denoted by $\{\omega_j\}_{j=1}^\infty$), we may assume $\omega_j \rightarrow \omega$ weakly star in  $L^\infty$, weakly in $L^p$ and weakly in $L^2$ as $j\to \infty$. Note that $G_\Pi(x,y)\in L_{loc}^{p'}(\overline{\Pi}\times\overline{\Pi})$. With similar arguments as in the proof of Lemma \ref{lm6}, it is easy to show that $\omega$ is indeed a maximizer. The proof is thus complete.		
\end{proof}

Next, we prove some essential properties of the maximizers. Let $\Sigma_\mu^{s,\Gamma}$ be the set of maximizers of $\mathcal{E}$ over $\mathcal{A}^{s,\Gamma}_{\mu}$. We first show that each maximizer is not a patch type if $\Gamma$ is sufficiently large.
\begin{lemma}\label{lm9}
	There exists $\Gamma_1>0$ such that for $\Gamma>\Gamma_1$, there holds
	\begin{equation*}
		\text{meas}\left(\{0<\omega<\Gamma\}\right)>0, \quad \forall\, \omega\in \Sigma_\mu^{s,\Gamma}.
	\end{equation*}
\end{lemma}
\begin{proof}
	We take  a maximizer  $\omega\in \Sigma_\mu^{s,\Gamma}$. Since $S_\mu^{s,\Gamma}>0$ for $\mu>0$, $\omega$ is not trivial. Assume that $\omega=\Gamma 1_\Omega$ for some measurable set $\Omega\subset \Pi$ with $\Gamma \text{meas}(\Omega) \leq 1$. We choose $r>0$ such that $\pi r^2=\text{meas}(\Omega)$. Then, by a simple rearrangement inequality, we have
	\begin{align*}
		\mathcal{E}(\omega)&=\frac{\Gamma^2}{2}\int_{\Omega}\int_{\Omega} G_\Pi(x,y)dxdy-\frac{\Gamma^2}{2}\text{meas}(\Omega) \\
		&\leq \frac{\Gamma^2}{2}\int_{\Omega}\int_{\Omega} \frac{c_{2,s}}{|x-y|^{2-2s}} dxdy -\frac{\Gamma^2}{2}\text{meas}(\Omega) \\
		&\leq \frac{\Gamma^2}{2}\int_{B_r(0)}\int_{B_r(0)} \frac{c_{2,s}}{|x-y|^{2-2s}}dxdy -\frac{\Gamma^2}{2}\text{meas}(\Omega) \\
		&=\frac{\Gamma^2\text{meas}(\Omega)}{2}\left(A[\text{meas}(\Omega)]^s-1 \right)\\
		&\leq \frac{\Gamma^2\text{meas}(\Omega)}{2}\left(\frac{A}{\Gamma^s}-1 \right),
	\end{align*}
	where $$A=\frac{1}{\pi^{1+s}}\int_{B_1(0)}\int_{B_1(0)} \frac{c_{2,s}}{|x-y|^{2-2s}}dxdy.$$ Taking  $\Gamma_1=\max\{A^{\frac{1}{s}},1\} $, we infer from the above inequality that $\mathcal{E}(\omega)<0$ for $\Gamma>\Gamma_1$. This is a contradiction. The proof is thus complete.
	
\end{proof}

By standard arguments, we can deduce the following relation between a maximizer $\omega$ and  its corresponding stream function.
\begin{lemma}\label{lm10}
	Suppose $\Gamma>\Gamma_1$, then for each $\omega\in \Sigma_\mu^{s,\Gamma}$, there exist constants $W,\gamma\geq 0$ such that
	\begin{equation}\label{2-12}
		\begin{cases}
			\mathcal{G}_s\omega-Wx_2-\gamma\leq 0,\quad &\text{on}\ \ \{\omega=0\},\\
			\mathcal{G}_s\omega-Wx_2-\gamma=\omega,\quad &\text{on}\ \ \{0<\omega<\gamma\},\\
			\mathcal{G}_s\omega-Wx_2-\gamma\geq \Gamma,\quad &\text{on}\ \ \{\omega=\Gamma\}.
		\end{cases}	
	\end{equation}
Moreover, $W$ and $\gamma$ are uniquely determined by $\omega$.
\end{lemma}

\begin{proof}
We will follow idea of Turkington \cite{T2} to prove Lemma \ref{lm10}.
	By the previous lemma, there exists $0<\delta_0<\Gamma$ such that $\text{meas}\left(\{\delta_0<\omega<\Gamma\}\right)>0$. We take functions $h_1,h_2\in L^\infty(\Pi)$ with compact support and satisfying
	\begin{align*}
		\begin{cases}
		\text{supp}(h_1), \text{supp}(h_2) \subset \{\omega\geq \delta_0\},\\
		\int_\Pi h_1(x)dx=1,\quad \int_\Pi x_2h_1(x)dx=0,\\
		\int_\Pi h_2(x)dx=0,\quad \int_\Pi x_2h_2(x)dx=1.
		\end{cases}
	\end{align*}
	For any $0<\delta<\delta_0$, we take a function $h \in L^\infty(\Pi)$ with compact support, $h\geq 0$ on $\{0\leq \omega\leq \delta\}$ and $h\leq 0$ on $\{\Gamma-\delta\leq \omega\leq \Gamma\}$.
	We set the test function $$\omega_\tau=\omega+\tau \eta,\ \ \tau>0,$$
	with $$\eta=\left(h-\left(\int_\Pi hdx\right)h_1-\left(\int_\Pi x_2hdx\right)h_2\right).$$
	One verify that $\omega_\tau\in \mathcal{A}^{s,\Gamma}_\mu$ if $\tau>0$ is sufficiently small. Since $\omega$ is a maximizer, there holds
	$$0\geq \frac{d\mathcal{E}(\omega_\tau)}{d\tau} \Bigg{|}_{\tau=0_+}=\int_\Pi (\mathcal{G}_s\omega-\omega)\eta dx. $$
 Set
 \begin{equation*}
   \gamma=\int_\Pi (\mathcal{G}_s\omega-\omega)h_1 dx,\ \ \ \ \  W=\int_\Pi (\mathcal{G}_s\omega-\omega)h_2 dx.
 \end{equation*}
We deduce that
\begin{equation*}
  \begin{split}
     0 & \geq \int_\Pi (\mathcal{G}_s\omega-\omega)\eta dx \\
       & = \int_\Pi (\mathcal{G}_s\omega-\omega)h dx -W\int_\Pi x_2hdx-\gamma\int_\pi hdx\\
       & =\int_\Pi \left(\mathcal{G}_s\omega-Wx_2-\gamma-\omega\right)hdx.
  \end{split}
\end{equation*}
By the arbitrariness of $h$, we must have
	\begin{equation*}
		\begin{cases}
			\mathcal{G}_s\omega-Wx_2-\gamma-\omega\leq 0,\quad\text{on}\,\,\{0\leq \omega< \delta\},\\
			\mathcal{G}_s\omega-Wx_2-\gamma-\omega=0,\quad\text{on}\,\,\{\delta\leq \omega\leq \Gamma-\delta\},\\
			\mathcal{G}_s\omega-Wx_2-\gamma-\omega\geq 0,\quad\text{on}\,\,\{\Gamma-\delta<\omega\leq\Gamma\},
		\end{cases}
	\end{equation*}
	which, by letting $\delta\rightarrow 0$, implies \eqref{2-12}.
	
	Since $\int_\Pi \omega dx\leq 1$, we can take a sequence $\{x^j\}_{j=1}^\infty$ with $x^j=(x^j_1, x^j_2)$, such that $x^j_1\rightarrow \infty$, $x^j_2\rightarrow 0$ and $\omega(x^j)\rightarrow 0$ as $j\to +\infty$. Thus, by \eqref{2-4} in Lemma \ref{lm2}, we have
	$$0=\lim_{j\rightarrow \infty} (\mathcal{G}_s\omega(x^j)-Wx^j_2-\gamma)_+=(-\gamma)_+,$$
	which implies that $\gamma\geq 0$.  Similarly, we can also take another sequence $\{\bar{x}^j\}_{j=1}^\infty$ with $\bar{x}^j=(\bar{x}^j_1, \bar{x}^j_2)$, such that $\bar{x}^j_1\rightarrow 0$, $\bar{x}^j_2\rightarrow \infty$ and $\omega(\bar{x}^j)\rightarrow 0$ as $j\to +\infty$. In this case, we have
	$$0=\lim_{j\rightarrow \infty} (\mathcal{G}_s\omega(\bar{x}^j)-W\bar{x}^j_2-\gamma)_+=\lim_{j\rightarrow \infty}(-W\bar{x}^j_2-\gamma)_+,$$
	which gives $W\geq 0$.
	
	Now, we show the uniqueness of $W$ and $\gamma$. Indeed, if there are $\hat{W}, \hat{\gamma}\geq 0$ such that \eqref{2-12} holds. Then, $$\mathcal{G}_s\omega(x)-\hat{W}x_2-\hat{\gamma}=\mathcal{G}_s\omega(x)-Wx_2-\gamma,$$ for all $x\in\Pi$ such that $0<\omega(x)<\Gamma$. This is, $$(\hat{W}-W)x_2=\gamma-\hat{\gamma},$$ for all $x\in\Pi$ such that $0<\omega(x)<\Gamma$, which leads to  $\hat{W}=W$ and $\hat{\gamma}=\gamma$, since $\{0<\omega<\Gamma\}$ has positive measure. The proof of Lemma \ref{lm10} is therefore finished.
\end{proof}

Thanks to Lemma \ref{lm10}, we show that all maximizers are uniformly bounded in $L^\infty$ with respect to $\Gamma$. This fact shows that $\Sigma_\mu^{s}=\Sigma_\mu^{s,\Gamma}$ if $\Gamma$ is large enough. In other words, we obtain the existence of maximizers for $\mathcal{E}$ over $\mathcal{A}^s_{\mu}$.
\begin{lemma}\label{lm11}
	There exists a constant $M_0>0$ such that, for all $\Gamma>\Gamma_1$, there holds
	\begin{equation}\label{2-13}
		0\leq \omega\leq M_0, \quad \forall\, \omega\in 	 \Sigma_\mu^{s, \Gamma}.
	\end{equation}
\end{lemma}
\begin{proof}
	Let $\omega\in 	 \Sigma_\mu^{s,\Gamma}$. By Lemma \ref{lm10}, we have
\begin{equation*}
  \omega(x)\leq (\mathcal{G}_s\omega(x)-Wx_2-\gamma)_+,\ \ \ \forall\,x\in\Pi,
\end{equation*}
for some $W,\gamma\geq 0$. Hence, by direct calculations, we derive
	\begin{align*}
		\omega(x)&\leq \left(\mathcal{G}_s\omega(x)-Wx_2-\gamma\right)_+\leq \mathcal{G}_s\omega(x)\\
		&\leq \int_\Pi \frac{c_{2,s}}{|x-y|^{2-2s}} \omega(y)dy\\
		&\leq \|\omega\|_\infty \int_{|x-y|<r} \frac{c_{2,s}}{|x-y|^{2-2s}} dy +\frac{c_{2,s}}{r^{2-2s}}\\
		&\leq C_1r^{2s}\|\omega\|_\infty +\frac{c_{2,s}}{r^{2-2s}}
	\end{align*}
for all $x\in\Pi$, where $C_1$ is a positive constant depending only on $s$.
Choosing $r$ such that $C_1r^{2s}={1}/{2}$, we infer from the above inequality that $\|\omega\|_\infty\leq M_0$ for some constant $M_0>0$ depending only on $s$. This completes the proof of Lemma \ref{lm11}.	
\end{proof}

As a consequence of Lemma \ref{lm11}, we have the following lemma.
\begin{lemma}\label{lmadd2}
  If $\Gamma>\Gamma_0:=\max\{\Gamma_1,M_0\}$, then $\Sigma_\mu^{s}=\Sigma_\mu^{s,\Gamma}$.
\end{lemma}

Similar to Lemma \ref{lmadd1}, we also have the following result.
\begin{lemma}\label{lmadd3}
  If $0<\mu_1<\mu_2<\infty$, then $S^s_{\mu_1}<S^s_{\mu_2}$.
\end{lemma}

\section{Uniqueness of maximizers}
In the previous section, we proved the existence of maximizers for $\mathcal{E}$ over $\mathcal{A}_\mu^s$. In this section, we show the uniqueness of maximizers. Arguing as in the proof of Lemma \ref{lm10}, we can establish the following result.
\begin{lemma}\label{lm12}
Let $0<s<1$. Suppose $\omega\in \Sigma^s_\mu$, then there exist two nonnegative constants $W$ and $\gamma$ such that
	\begin{equation}\label{3-1}
		\omega=\left(\mathcal{G}_s\omega-Wx_2-\gamma\right)_+.
	\end{equation}
	Moreover, $W$ and $\gamma$ are uniquely determined by $\omega$.
\end{lemma}

Note that if $\omega\in \Sigma^s_\mu$ and $\gamma>0$, then $\text{supp}(\omega)$ and $x_1$-axis have a positive distance. The following lemma shows that if
$\mu$ is sufficiently small, then $W>0$ and $\gamma=0$.
\begin{lemma}\label{lm13}
There exists a constant $\mu_0>0$ such that if $0<\mu\leq \mu_0$, then the constants $W, \gamma$ in Lemma \ref{lm12} satisfy $W>0$ and $\gamma=0$.
\end{lemma}
\begin{proof}
Let $\omega\in \Sigma^s_\mu$. We first prove $\gamma=0$ for small $\mu$. Since
\begin{equation*}
  \mu=\int_\Pi x_2\omega dx \geq 2\mu\int_{x_2\geq 2\mu} \omega dx,
\end{equation*}
we have
	\begin{equation}\label{3-2}
		\int_{x_2\geq2\mu} \omega dx\leq \frac{1}{2}.
	\end{equation}
	Now, we estimate $\int_{0<x_2<2\mu} \omega dx$.
	\begin{align}\label{3-3}
		\int_{0<x_2<2\mu} \omega dx&=\int_{0<x_2<2\mu} \left(\mathcal{G}_s\omega-Wx_2-\gamma\right)_+dx\leq \int_{0<x_2<2\mu} \mathcal{G}_s\omega dx\nonumber\\
		& \leq \int_{0<x_2<4\mu}\omega(x)\int_{0<y_2<2\mu} G_\Pi(x,y)dydx+ \int_{x_2\geq 4\mu}\omega(x)\int_{0<y_2<2\mu} G_\Pi(x,y)dydx.
	\end{align}
	On the one hand, for $x_2\geq 4\mu$ ,
	\begin{align*}
		\int_{0<y_2<2\mu} G_\Pi(x,y)dy
		&=\int_{0<y_2<2\mu} c_{2,s}\left(\frac{1}{|x-y|^{2-2s}}-\frac{1}{|x-\bar{y}|^{2-2s}}\right)dy\\
		&\leq \int_{0<y_2<2\mu} \frac{4c_{2,s}x_2y_2}{|x-y|^{4-2s}}dy\\
		&\leq Cx_2 \mu^{2s-1}.
	\end{align*}
	Hence, we have
	\begin{equation}\label{3-4}
		\int_{x_2\geq 4\mu}\omega(x)\int_{0<y_2<2\mu} G_\Pi(x,y)dydx\leq C\mu^{2s-1}\int_\Pi x_2\omega dx=C\mu^{2s}.
	\end{equation}
	On the other hand, for $0<x_2<4\mu$, we estimate
	\begin{align*}
		\int_{0<y_2<2\mu} G_\Pi(x,y)dy&=\int_{0<y_2<2\mu} c_{2,s}\left(\frac{1}{|x-y|^{2-2s}}-\frac{1}{|x-\bar{y}|^{2-2s}}\right)dy\\
		&\leq \int_{0<y_2<2\mu,|x-y|<8\mu} \frac{c_{2,s}}{|x-y|^{2-2s}}dy+\int_{0<y_2<2\mu,|x-y|\geq 8\mu} \frac{4c_{2,s} x_2y_2}{|x-y|^{4-2s}}dy\\
		&\leq \int_{|x-y|<8\mu} \frac{c_{2,s}}{|x-y|^{2-2s}}dy+C\int_{|x_1-y_1|\geq 2\mu}\int_{0<y_2<2\mu}\frac{x_2y_2}{|x_1-y_1|^{4-2s}}dy_1dy_2\\
		&\leq C( \mu^{2s}+x_2 \mu^{2s-1}),
	\end{align*}
	which implies
	\begin{equation}\label{3-5}
		\int_{0<x_2< 4\mu}\omega(x)\int_{0<y_2<2\mu} G_\Pi(x,y)dydx\leq C\mu^{2s}\int_\Pi \omega dx+C\mu^{2s-1}\int_\Pi x_2\omega dx\leq C\mu^{2s}.
	\end{equation}
	Combining \eqref{3-2}, \eqref{3-3}, \eqref{3-4} and \eqref{3-5}, we derive
	\begin{equation}\label{3-6}
		\int_\Pi \omega dx\leq \frac{1}{2}+C\mu^{2s}.
	\end{equation}	
	Hence, there exists $\mu_0>0$ small such that for $0<\mu\leq \mu_0$, $$\int_\Pi \omega dx<1.$$
	If $\int_\Pi \omega dx<1$, we can take $$\eta=h-\left(\int_\Pi x_2hdx\right)h_2.$$ And taking $\omega+\tau \eta$ for sufficiently small $\tau>0$ as the test function in the proof of Lemma \ref{lm10}, we conclude $$\omega=(\mathcal{G}_s\omega-Wx_2)_+, $$ which means $\gamma=0$.
	
	Next, we prove $W>0$. Indeed, by \eqref{3-1}, we have
	\begin{align*}
		0&<\int_\Pi \omega \mathcal{G}_s\omega dx -\int_\Pi \omega^2 dx\\
		&=\int_\Pi \omega \mathcal{G}_s\omega dx -\int_\Pi \omega(\mathcal{G}_s\omega -Wx_2)_+ dx\\
		&=\int_\Pi \omega \mathcal{G}_s\omega dx -\int_\Pi \omega(\mathcal{G}_s\omega  -Wx_2) dx\\
		&=W\mu.
	\end{align*}
Thus, we conclude $W>0$ and the proof of Lemma \ref{lm13} is completed.
\end{proof}

The property $W>0$ or $\gamma>0$ implies that the support of a maximizer is a bounded set.
\begin{lemma}\label{lm14}
Suppose $\omega\in \Sigma^s_\mu$, then $\text{supp}(\omega)$ is a bounded set in $\Pi$.
\end{lemma}
\begin{proof}
	Let $\omega\in \Sigma^s_\mu$. Then, by \eqref{3-1}, we have $\omega=(\mathcal{G}_s\omega-Wx_2-\gamma)_+$. We first consider the case $\gamma>0$. In this case, the conclusion follows easily from \eqref{2-4}. If $\gamma=0$, we can conclude from the proof of Lemma \ref{lm13} that $W>0$. By \eqref{2-1} and \eqref{3-1}, we have $\omega \in L^1(\Pi)\cap L^\infty(\Pi)$. We extend $\omega=(\mathcal{G}\omega-Wx_2)_+$ and the stream function $\psi=\mathcal{G}\omega$ to the whole space.
	Let $$\bar{\omega}(x)=\begin{cases} \omega(x),\quad &\text{if}\ \  x_2>0,\\ 0,\quad &\text{if}\ \ x_2=0, \\ -\omega(x_1,-x_2),\quad &\text{if}\ \ x_2<0,\end{cases}$$
	and
	$$\bar{\psi}(x)=\begin{cases} \psi(x),\quad &\text{if}\ \ x_2>0,\\ 0,\quad &\text{if}\ \ x_2=0, \\ -\psi(x_1,-x_2),\quad &\text{if}\ \ x_2<0.\end{cases}$$
	Then,
\begin{equation*}
  \bar{\psi}(x)=\int_{\mathbb{R}^2} \frac{c_{2,s}}{|x-y|^{2-2s}}\bar{\omega}(y)dy.
\end{equation*}
Since $ \bar{\omega}\in L^1(\Pi)\cap L^\infty(\mathbb{R}^2)$, we deduce $\bar{\psi}\in W^{2s,p}(\mathbb{R}^2)$ for all $2\leq p<\infty$ by the property of Riesz transformation (see Chapter 5 of  \cite{S70}).  Then, we have $||\bar{\psi}||_{C^\alpha(\mathbb{R}^2)}\leq C$ for some $\alpha>0$ small by Sobolev embedding theorem. By $\omega=(\mathcal{G}\omega-Wx_2)_+$ and definitions of $\bar{\psi}(x)$ and $\bar{\omega}$, we have $\|\bar{\omega}\|_{C^\alpha(\mathbb{R}^2)}\leq C$. We further infer from the regularity theory for fractional equations (see, e.g., Propositions 2.8 and 2.9 in \cite{S07}) and bootstrap argument that $\|\bar{\psi}\|_{C^{1,\alpha}(\mathbb{R}^2)}\leq C$ for some small $\alpha>0$. Thus, $\|\nabla \bar{\psi}\|_{C^{\alpha}(\mathbb{R}^2)}\leq C$.
To prove that the support of $\omega=(\psi-Wx_2)_+$ is bounded, we estimate ${\psi(x)}/{x_2}$. Observing that $$\frac{\psi(x)}{x_2}=\int_0^1 \partial_2 \psi (x_1, \tau x_2)d\tau,$$ we have $\|{\psi(x)}/{x_2}\|_{C^{\alpha}(\overline{\Pi})}\leq C$. By Hardy's inequality, we deduce $\|{\psi(x)}/{x_2}\|_q\leq \|\nabla\psi\|_q<\infty$ for $2\leq q<\infty$. Thus, we must have $$\frac{\psi(x)}{x_2}\rightarrow 0,\quad\text{as}\ \ |x|\rightarrow \infty,$$ which implies the boundedness of $\text{supp}(\omega)$ and completes the proof of this Lemma.
\end{proof}

The uniqueness of Hill's vortex and the Chaplygin-Lamb dipole was proved by \cite{AF86} and \cite{B96}, respectively. Their method essentially relies on a transform of the differential equations of stream function and the method of moving planes. In our case, similar transforms are hard to be established for differential equation because of nonlocal nature of fractional Laplacians. Instead, we deal with the integral equation directly. Fortunately, the integral equation $\omega=(\mathcal{G}_s\omega-Wx_2)_+$ in the half plane $\Pi$ can be reduced to an integral equation in $\mathbb{R}^4$ by a suitable transform. And hence, we can apply the method of moving planes and Theorem 1.1 in \cite{C20} to prove the uniqueness of maximizers, which completes the proof of Theorem \ref{Uq}.
\begin{lemma}\label{lm15}
	Let $0<s<1$ and $0<\mu\leq \mu_0$. Then there exists a $\omega_\mu^s
\in \mathcal{A}_\mu^s$ such that $$\Sigma^s_\mu=\{\omega^s_\mu(\cdot+c\mathbf{e}_1)\mid c\in\mathbb{R}\}.$$
Moreover, $\text{supp}(\omega_\mu^s)=\overline{B_{r_\mu^s}(0)}\cap \Pi$ for some $r_\mu^s>0$.
\end{lemma}
\begin{proof}
Let $\omega\in \Sigma^s_\mu$ and $\psi=\mathcal{G}_s\omega$. For $y=(y',y_4)\in \mathbb{R}^4$, we set
\begin{equation*}
  \phi(y)=\frac{\psi(y_4, |y'|)}{|y'|}.
\end{equation*}
Since $\omega$ has bounded support by Lemma \ref{lm14}. For $|x|$ large, it is easy to see $$\frac{\psi(x)}{x_2}=O\left(\frac{1}{|x|^{4-2s}}\right).$$ Hence $$\phi(y)=O\left(\frac{1}{|y|^{4-2s}}\right).$$
	By a direction calculation, we have
	\begin{align*}
		&\quad \int_{\mathbb{R}^4} \frac{c_{4,s}}{|x-y|^{4-2s}} (\phi(y)-W)_+dy\\
		&=\int_{\mathbb{R}}\int_{\mathbb{R}^3} \frac{c_{4,s}}{|x-y|^{4-2s}|y'|} (\psi(y_4,|y'|)-W|y'|)_+dy'dy_4\\
		&=2\pi c_{4,s}\int_{\mathbb{R}}\int_0^\infty \int_0^\pi \frac{\rho\sin\theta(\psi(y_4,\rho)-W\rho)_+}{(\rho_x^2+\rho^2-2\rho_x\rho\cos\theta+(x_4-y_4)^2)^{2-s}}d\theta d\rho dy_4\ \ \ (\text{with}\ \rho_x:=|x|)\\
		&=\frac{\pi c_{4,s}}{(1-s)\rho_x}\int_{\mathbb{R}}\int_0^\infty  \left(\frac{1}{((\rho_x-\rho)^2+(x_4-y_4)^2)^{1-s}}-\frac{1}{((\rho_x+\rho)^2+(x_4-y_4)^2)^{1-s}}\right) \\
		&\quad\cdot (\psi(y_4,\rho)-W\rho)_+ d\rho dy_4\\
		&=\frac{\psi(x_4, |x'|)}{|x'|}\\
		&=\phi(x).
	\end{align*}
	Thus, $\phi$ satisfies the integral equation
	\begin{equation}\label{3-7}
		\phi(x)=\int_{\mathbb{R}^4} \frac{c_{4,s}}{|x-y|^{4-2s}} (\phi(y)-W)_+dy.	
	\end{equation}
	Since $\phi$ is continuous and the support of $(\phi(y)-W)_+$ is compact, one can apply the method of moving planes to deduce that $\phi$ is radially symmetric with respect to some point $y^0=(0,y^0_4)$ and hence unique up to translations in $y_4$ by Theorem 1.1 of \cite{C20}. Therefore, there exists a unique function $\omega^s_\mu\in \Sigma^s_\mu$, whose support is a half disk centered at the origin. The proof of Lemma \ref{lm15} is hence completed.
\end{proof}

We are now in a position to prove Theorem \ref{Ee}.
\begin{proof}[Proof of Theorem \ref{Ee}]
Let $0<s<1$ and $0<\mu\leq \mu_0$. Then Theorem \ref{Ee} follows from the above lemmas by letting
$$\omega_L(x)=\begin{cases}~~~\omega^s_\mu(x),&\text{if}\ \,x_2>0,\\ -\omega^s_\mu(x_1,-x_2),&\text{if}\  \,x_2\leq 0. \end{cases}$$
\end{proof}

\section{compactness of maximizing sequences}
In this section, we show the compactness of a general maximizing sequence up to translation for the $x_1$-variable by using a concentration compactness principle. Recall that $p_s=\infty$ if $0<s\leq {1}/{2}$ and $p_s=2$ if ${1}/{2}<s<1$.
\begin{theorem}\label{cp}
Let $0<s<1$ and $0<\mu\leq \mu_0$. Suppose that $\{\omega_n\}_{n=1}^\infty$ is a maximizing sequence in the sense that
	\begin{equation}\label{4-1}
		\omega_n\geq 0,\  \omega_n\in L^1(\Pi)\cap L^{p_s}(\Pi),\ \int_\Pi \omega_n dx\leq 1,\ \|\omega_n\|_{p_s}\leq C,\ \ \forall \, n\geq 1,
	\end{equation}

	\begin{equation}\label{4-2}
		\mu_n=\int_\Pi x_2\omega_n dx\rightarrow \mu, \quad \text{as}\ \ n\rightarrow \infty,
	\end{equation}
	and
	\begin{equation}\label{4-3}
		\mathcal{E}(\omega_n)\rightarrow S^s_{\mu},\quad \text{as}\ \ n\rightarrow \infty.
	\end{equation}
	Then, there exist $\omega\in \Sigma^s_{\mu}$, a subsequence $\{\omega_{n_k}\}_{k=1}^\infty$ and a sequence of real number $\{c_k\}_{k=1}^\infty$ such that
	\begin{equation}\label{4-4}
		\omega_{n_k}(\cdot+c_k\mathbf{e}_1)\rightarrow \omega\quad\text{strongly in}\,\, L^2(\Pi),
	\end{equation}
	and
	\begin{equation}\label{4-5}
		x_2\omega_{n_k}(\cdot+c_k\mathbf{e}_1)\rightarrow x_2\omega\quad\text{strongly in}\,\, L^1(\Pi),
	\end{equation}
	as $k\rightarrow \infty$.
\end{theorem}
To prove Theorem \ref{cp}, we need the concentration compactness lemma, which is due to Lions \cite{L84}.
\begin{lemma}\label{ccp}
	Let $\{\xi_n\}_{n=1}^\infty$ be a sequence of nonnegative functions in $L^1(\Pi)$ satisfying
	$$\limsup_{n\rightarrow \infty} \int_\Pi \xi_n dx\rightarrow \mu,$$ for some $0< \mu<\infty$.
	Then, after passing to a subsequence, one of the following holds:\\
	(i) (Compactness) There exists a sequence $\{y_n\}_{n=1}^\infty$ in $\overline{\Pi}$ such that for arbitrary $\varepsilon>0$, there exists $R>0$ satisfying
	\begin{equation*}
		\int_{\Pi\cap B_R(y_n)}\xi_n dx\geq \mu-\varepsilon, \quad \forall n\geq 1.
	\end{equation*}\\
	(ii) (Vanishing) For each $R>0$,
	\begin{equation*}
		\lim_{n\rightarrow \infty}\sup_{y\in \Pi}  \int_{B_R(y)\cap\Pi} \xi_n dx =0.
	\end{equation*} \\
	(iii) (Dichotomy) There exists a constant $0<\alpha<\mu$ such that for any $\varepsilon>0$, there exist $N=N(\varepsilon)\geq 1$ and $0\leq \xi_{i,n}\leq \xi_n, \,i=1,2$ satisfying
	\begin{equation*}
		\begin{cases}
			\|\xi_n-\xi_{1,n}-\xi_{2,n}\|_1+|\alpha-\int_\Pi \xi_{1,n} dx|+|\mu-\alpha-\int_\Pi \xi_{2,n} dx|<\varepsilon,\quad \text{for}\,\,n\geq N,\\
			d_n:=dist(supp(\xi_{1,n}), supp(\xi_{2,n}))\rightarrow \infty, \quad \text{as}\,\,n\rightarrow \infty.
		\end{cases}	
	\end{equation*}
\end{lemma}
\begin{proof}
	This lemma is a slight reformulation of Lemma 1.1 \cite{L84}, so we omit the proof.
\end{proof}

\begin{proof}[Proof of Theorem \ref{cp}]
Let $\xi_n=x_2\omega_n$ and apply Lemma \ref{ccp}. Then, for a certain subsequence, still denoted by $\{\omega_n\}_{n=1}^\infty$, one of the three cases in Lemma \ref{ccp} should occur. We divide the proof into three steps.

\emph{Step 1. Vanishing excluded:}
Suppose for each fixed $R>0$,
\begin{equation}\label{4-6}
	\lim_{n\rightarrow \infty}\sup_{y\in \Pi}  \int_{B_R(y)\cap\Pi} x_2\omega_n dx =0.
\end{equation}
We will show $\lim_{n\rightarrow \infty} E(\omega_n)=0$, which contradicts to the fact $S^s_\mu>0$. Recall that
\begin{equation*}
  G_\Pi(x,y)\leq \frac{4x_2y_2}{|x-y|^{4-2s}}.
\end{equation*}
Hence $G_\Pi(x,y)\geq Rx_2y_2$ implies $|x-y|\leq CR^{-\frac{1}{4-2s}}$. By $G_\Pi(x,y)\in L_{loc}^{p_s'}(\overline{\Pi}\times\overline{\Pi})$, we estimate
\begin{align*}
	&\quad2E(\omega_n)\\
	&=\int_\Pi\int_\Pi \omega_n(x)G_\Pi(x,y)\omega_n(y)dxdy\\
	&=\iint_{|x-y|\geq R}\omega_n(x)G_\Pi(x,y)\omega_n(y)dxdy  +\iint_{|x-y|< R}\omega_n(x)G_\Pi(x,y)\omega_n(y)dxdy\\
	&\leq \frac{C\mu^2}{R^{4-2s}}+ \iint_{|x-y|< R,\atop G_\Pi(x,y)<Rx_2y_2} \omega_n(x)G_\Pi(x,y)\omega_n(y)dxdy \\
&\ \ \ \ \ \ \ \ \ \ \ \ \ \ \ \ \ \ \ \ \ \ \ \ +\iint_{|x-y|< R, \atop G_\Pi(x,y)\geq Rx_2y_2} \omega_n(x)G_\Pi(x,y)\omega_n(y)dxdy\\
	&\leq  \frac{C\mu^2}{R^{4-2s}}+R\mu\left(\sup_{y\in \Pi}  \int_{B_R(y)\cap\Pi} x_2\omega_n dx\right) +\iint_{|x-y|< CR^{-\frac{1}{4-2s}}} \omega_n(x)G_\Pi(x,y)\omega_n(y)dxdy\\
	&\leq  \frac{C\mu^2}{R^{4-2s}}+R\mu\left(\sup_{y\in \Pi}  \int_{B_R(y)\cap\Pi} x_2\omega_n dx\right) +\|\omega_n\|_{p_s}\int_\Pi \omega_n\left(\int_{|x-y|< CR^{-\frac{1}{4-2s}}} G_\Pi(x,y)^{p'_s}dy\right)^{\frac{1}{p'_s}}dx\\
& \leq  \frac{C\mu^2}{R^{4-2s}}+R\mu\left(\sup_{y\in \Pi}  \int_{B_R(y)\cap\Pi} x_2\omega_n dx\right) +CR^{-\frac{s-{1}/{p_s}}{2-s}}.
\end{align*}
We infer from the above calculations by first letting $n\rightarrow \infty$, then $R\rightarrow \infty$ that
$$\lim_{n\rightarrow \infty} E(\omega_n)=0.$$
However,  $0<S_\mu=\lim_{n\rightarrow \infty} \mathcal{E}(\omega_n) \leq \lim_{n\rightarrow \infty} E(\omega_n)$. This is a contradiction.

\emph{Step 2. Dichotomy excluded:}
Suppose there exists a constant $0<\alpha<\mu$ such that for any $\varepsilon>0$, there exist $N(\varepsilon)\geq 1$ and $0\leq \omega_{i,n}\leq \omega_n, \,i=1,2,3$ satisfying
\begin{equation*}
	\begin{cases}
		\omega_n=\omega_{1,n}+\omega_{2,n}+\omega_{3,n},\\
		||x_2\omega_{3,n}||_1+|\alpha-\alpha_n|+|\mu-\alpha-\beta_n|<\varepsilon,\quad \text{for}\,\,n\geq N(\varepsilon),\\
		d_n:=dist(supp(\omega_{1,n}), supp(\omega_{2,n}))\rightarrow \infty, \quad \text{as}\,\,n\rightarrow \infty,
	\end{cases}	
\end{equation*}
where $\alpha_n=\|x_2\omega_{1,n}\|_1$ and $ \beta_n=\|x_2\omega_{2,n}\|_1$. Using the diagonal argument, we obtain that there exists a subsequence, still denoted by $\{\omega_n\}_{n=1}^\infty$, such that
\begin{equation}\label{4-7}
	\begin{cases}
		\omega_n=\omega_{1,n}+\omega_{2,n}+\omega_{3,n}, \quad 0\leq \omega_{i,n}\leq \omega_n, \,i=1,2,3\\
		\|x_2\omega_{3,n}\|_1+|\alpha-\alpha_n|+|\mu-\alpha-\beta_n|\rightarrow 0,\quad \text{as}\,\,n\rightarrow \infty,\\
		d_n:=dist(supp(\omega_{1,n}), supp(\omega_{2,n}))\rightarrow \infty, \quad \text{as}\,\,n\rightarrow \infty.
	\end{cases}	
\end{equation}
By the symmetry of $E$, we have
\begin{align*}
	2E(&\omega_n)=E(\omega_{1,n}+\omega_{2,n}+\omega_{3,n})\\
	&=\int_\Pi\int_\Pi \omega_{1,n}(x)G_\Pi(x,y)\omega_{1,n}(y)dxdy+\int_\Pi\int_\Pi \omega_{2,n}(x)G_\Pi(x,y)\omega_{2,n}(y)dxdy\\
	&\quad +2\int_\Pi\int_\Pi \omega_{1,n}(x)G_\Pi(x,y)\omega_{2,n}(y)dxdy+\int_\Pi\int_\Pi (2\omega_n-\omega_{3,n}(x))G_\Pi(x,y)\omega_{3,n}(y)dxdy
\end{align*}
For fixed $M>0$,
\begin{align*}
	 &\int_\Pi\int_\Pi (2\omega_n-\omega_{3,n}(x))G_\Pi(x,y)\omega_{3,n}(y)dxdy\\
	&\leq \iint_{G_\Pi(x,y)< Mx_2y_2}2\omega_n(x)G_\Pi(x,y)\omega_{3,n}(y)dxdy +\iint_{G_\Pi(x,y)\geq  Mx_2y_2}2\omega_n(x)G_\Pi(x,y)\omega_{3,n}(y)dxdy\\
	&\leq 2M\mu_n \|x_2\omega_{3,n}\|_1+2\|\omega_{3,n}\|_{p_s}\int_\Pi \omega_n(x)\left(\int_{|x-y|< CM^{-\frac{1}{4-2s}}} G_\Pi(x,y)^{p'_s}dy\right)^{\frac{1}{p'_s}}dx\\
	&\leq 2Mo_n(1)+CM^{-\frac{s-{1}/{p_s}}{2-s}},
\end{align*}
It is obvious that
\begin{align*}
	\int_\Pi\int_\Pi \omega_{1,n}(x)G_\Pi(x,y)\omega_{2,n}(y)dxdy\leq \frac{C\mu_n^2}{d_n^{4-2s}}.
\end{align*}
Hence, we arrive at
$$\mathcal{E}(\omega_n)=E(\omega_n)-\frac{1}{2}\int_\Pi \omega^2_n dx \leq \mathcal{E}(\omega_{1,n})+\mathcal{E}(\omega_{2,n})+\frac{C\mu_n^2}{d_n^{4-2s}}+2Mo_n(1)+CM^{-\frac{s-{1}/{p_s}}{2-s}}.$$
Taking Steiner symmetrization $\omega^*_{i,n}$ of $\omega_{i,n}$ for $i=1,2$, we have
\begin{equation*}
	\begin{cases}
		\mathcal{E}(\omega_n)\leq \mathcal{E}(\omega^*_{1,n})+\mathcal{E}(\omega^*_{2,n})+\frac{C\mu_n^2}{d_n^{4-2s}}+2Mo_n(1)+CM^{-\frac{s-{1}/{p_s}}{2-s}},\\
		\|\omega^*_{1,n}\|_1+\|\omega^*_{2,n}\|_1\leq 1, \|\omega^*_{1,n}\|_{p_s}+\|\omega^*_{2,n}\|_{p_s}\leq C,\\
		\|x_2\omega^*_{1,n}\|_1=\alpha_n,   \|x_2\omega^*_{2,n}\|_1=\beta_n.
	\end{cases}	
\end{equation*}
We may assume that $\omega^*_{i,n}\rightarrow \omega^*_{i}$   weakly in $L^2(\Pi)$ for $i=1,2$ when ${1}/{2}<s<1$ and weakly star in $L^\infty$ when $0<s\leq {1}/{2}$. Then, arguing as in the proof of Lemma \ref{lm6}, we have the convergence of the kinetic energy
$$\lim_{n\rightarrow \infty} E(\omega^*_{i,n})=E(\omega^*_{i}),$$
for $i=1,2$.
First letting $n\rightarrow \infty$, then $M\rightarrow \infty$, we obtain
\begin{equation*}
	\begin{cases}
		S^s_\mu\leq \mathcal{E}(\omega^*_{1})+\mathcal{E}(\omega^*_{2}),\\
		\|\omega^*_{1}\|_1+\|\omega^*_{2}\|_1\leq 1, \|\omega^*_{1}\|_2+\|\omega^*_{2}\|_2\leq C,\\
		\|x_2\omega^*_{1}\|_1\leq\alpha,   \|x_2\omega^*_{2}\|_1\leq \mu-\alpha.
	\end{cases}	
\end{equation*}
Set $\alpha^*=\|x_2\omega^*_{1}\|_1\leq\alpha$, $\nu_1=\|\omega^*_{1}\|_1$,  $\beta^*=\|x_2\omega^*_{2}\|_1\leq\mu-\alpha$ and  $\nu_2=\|\omega^*_{2}\|_1$. We claim that $$\alpha^*, \beta^*>0.$$
In fact, suppose that $\alpha^*=0$, then we have $\omega_1^*\equiv0$, and hence $$S^s_\mu\leq \mathcal{E}(\omega^*_{1})+\mathcal{E}(\omega^*_{2})\leq \mathcal{E}(\omega^*_{2})\leq S_{\beta^*}^s.$$
This is a contradiction to Lemmas \ref{lmadd1} and \ref{lmadd3}. Similarly, one can verify $\beta^*>0$.
We choose $\hat{\omega}_1\in \Sigma^s_{\alpha^*, \nu_1}$, $\hat{\omega}_2\in \Sigma^s_{\beta^*,\nu_2}$. We further have that supports of $\hat{\omega}_i,$ $i=1,2$ are bounded by Lemma \ref{lm14}. Therefore, we may assume that $\text{supp}(\hat{\omega}_1)\cap \text{supp}(\hat{\omega}_2)=\varnothing$ by suitable translations in $x_1$-direction. Now, we set $\hat{\omega}=\hat{\omega}_1+\hat{\omega}_2$, then
\begin{equation*}
	\begin{cases}
		\int_\Pi \hat{\omega} dx=\int_\Pi\hat{\omega}_1 dx+\int_\Pi\hat{\omega}_2dx\leq 1,\\
		\int_\Pi x_2\hat{\omega} dx=\int_\Pi\hat{\omega}_1dx+\int_\Pi\hat{\omega}_2dx=\alpha^*+\beta^*\leq \mu,\\
	\end{cases}
\end{equation*}
which implies that $\hat{\omega}\in \mathcal{A}^s_{\alpha^*+\beta^*}$.
Thus, observing that $\hat{\omega}_1\not\equiv0$ and $\hat{\omega}_2\not\equiv0$, we have
\begin{equation*}
\begin{split}
   S^s_\mu & \leq \mathcal{E}(\omega^*_{1})+\mathcal{E}(\omega^*_{2})\leq \mathcal{E}(\hat{\omega}_1)+\mathcal{E}(\hat{\omega}_2)=\mathcal{E}(\hat{\omega})-2\int_\Pi\int_\Pi \hat{\omega}_1(x)G_\Pi(x,y)\hat{\omega}_2(y)dxdy \\
     & <S^s_{\alpha^*+\beta^*}\leq S^s_\mu,
\end{split}
\end{equation*}
which is a contradiction.

\emph{Step 3. Compactness:} Assume that there exists a sequence $\{y_n\}_{n=1}^\infty$ in $\overline{\Pi}$ such that for arbitrary $\varepsilon>0$, there exists $R>0$ satisfying
\begin{equation}\label{4-8}
	\int_{\Pi\cap B_R(y_n)} x_2\omega_n dx\geq \mu-\varepsilon, \quad \forall\, n\geq 1.
\end{equation}
We may assume that $y_n=(0, y_{n,2})$ after a suitable $x_1$-translation.
We claim:
\begin{equation}\label{4-9}
	\sup_{n\geq 1} y_{n,2}<\infty.
\end{equation}
Suppose on the contrary that there exists a subsequence, still denoted by $y_{n,2}$, such that $$\lim_{n\rightarrow \infty} y_{n,2}=\infty.$$
In this case, we calculate,
\begin{align*}
	2E(\omega_n)&=\int_\Pi\omega_n(x)\mathcal{G}_s\omega_n(x) dx\\
	&=\int_{\Pi\cap B_R(y_n)}\omega_n(x)\mathcal{G}_s\omega_n(x) dx +\int_{\Pi\setminus B_R(y_n)}\omega_n(x)\mathcal{G}_s\omega_n(x) dx.
\end{align*}
Recall that $\{\omega_n\}_{n=1}^\infty$ is uniformly bounded in $L^{p_s}(\Pi)$. By \eqref{2-3}, choosing $1/s<q\le p_s$ with $2s-2/q<1/2$, we derive
\begin{align*}
	\int_{\Pi\cap B_R(y_n)}\omega_n(x)\mathcal{G}_s\omega_n(x) dx
	\leq \frac{C}{(y_{n,2}-R)^{1+\frac{2}{q}-2s}}+\frac{C}{(y_{n,2}-R)^{4-2s}}\leq \frac{C}{(y_{n,2}-R)^{\frac{1}{2}}}
\end{align*}
Fixed $M>0$ large, we have
\begin{align}\label{4-10}
	&\quad\int_{\Pi\setminus B_R(y_n)}\omega_n(x)\mathcal{G}_s\omega_n(x) dx\\
	&=\iint_{x\in\Pi\setminus B_R(y_n),\atop  G_\Pi(x,y)\leq Mx_2y_2} \omega_n(x)G_\Pi(x,y)\omega_n(y) dxdy +\iint_{x\in\Pi\setminus B_R(y_n),\atop  G_\Pi(x,y)> Mx_2y_2} \omega_n(x)G_\Pi(x,y)\omega_n(y) dxdy\nonumber\\
	&\leq M\mu_n (\mu_n-\mu+\varepsilon) +CM^{-\frac{s-{1}/{p_s}}{2-s}}.\nonumber
\end{align}
Therefore, letting $n\rightarrow \infty$, $\varepsilon\rightarrow 0$ and $M\rightarrow \infty$ one by one, we conclude
$$0<S^s_\mu\leq \lim_{n\rightarrow\infty} E(\omega_n)=0,$$ which is absurd. Hence, we have proved claim \eqref{4-9}. Then, we may assume that $y_{n,2}=0$ by taking $R$ larger.
Now, we have
$$\int_{\Pi\cap B_R(0)} x_2\omega_n dx \geq \mu-\varepsilon, \quad \forall\, n\geq 1.$$
Since $\{\omega_n\}_{n=1}^\infty$ is uniformly bounded in $L^1\cap L^{p_s}$, we may assume that $\omega_n\rightarrow \omega$ weakly in $L^2$ and weakly-star in $L^{p_s}$. It is obvious that
$$\int_\Pi \omega dx\leq 1, \quad \int_\Pi x_2\omega dx=\mu,$$
and hence $\omega\in \mathcal{A}^s_\mu$.

Next, we prove the convergence of the energy.
On the one hand, for fixed $M>0$,  similar to \eqref{4-10}, we compute
\begin{align*}
	2E(\omega_n)&=\int_\Pi\int_\Pi \omega_n(x)G_\Pi(x,y)\omega_n(y)dxdy\\
	&\leq \int_{\Pi\cap B_R(0)}\int_{\Pi\cap B_R(0)} \omega_n(x)G_\Pi(x,y)\omega_n(y)dxdy +\int_{\Pi\setminus B_R(0)}\int_{\Pi} \omega_n(x)G_\Pi(x,y)\omega_n(y)dxdy\\
	&\quad+ \int_{\Pi}\int_{\Pi\setminus B_R(0)} \omega_n(x)G_\Pi(x,y)\omega_n(y)dxdy\\
	&\leq \int_{\Pi\cap B_R(0)}\int_{\Pi\cap B_R(0)} \omega_n(x)G_\Pi(x,y)\omega_n(y)dxdy +2\int_{\Pi\setminus B_R(0)}\omega_n(x)\mathcal{G}_s\omega_n(x) dx\\
	&\leq \int_{\Pi\cap B_R(0)}\int_{\Pi\cap B_R(0)} \omega_n(x)G_\Pi(x,y)\omega_n(y)dxdy +M\mu_n (\mu_n-\mu+\varepsilon) +CM^{-\frac{s-{1}/{p_s}}{2-s}}.
\end{align*}
Letting $n\rightarrow \infty$, $\varepsilon\rightarrow 0$ and $M\rightarrow \infty$ one by one, we deduce
$$\limsup_{n\rightarrow \infty} E(\omega_n)\leq E(\omega).$$
On the other hand, for any $R_1>0$, one has
$$	2E(\omega_n)=\int_\Pi\int_\Pi \omega_n(x)G_\Pi(x,y)\omega_n(y)dxdy\geq \int_{\Pi\cap B_{R_1}(0)} \int_{\Pi\cap B_{R_1}(0)}\omega_n(x)G_\Pi(x,y)\omega_n(y)dxdy,$$
which implies $$\liminf _{n\rightarrow \infty}E(\omega_n)\geq E(\omega)$$ by first letting $n\rightarrow \infty$ and then $R_1\rightarrow \infty$.
Therefore, we obtain the convergence
\begin{equation}\label{4-11}
	\lim _{n\rightarrow \infty }E(\omega_n)= E(\omega).
\end{equation}
Since $\omega\in \mathcal{A}_\mu^s$, we infer from \eqref{4-11} and the weak convergence that
\begin{equation}\label{4-12}
	S_\mu^s=\lim_{n\rightarrow \infty} \mathcal{E}(\omega_n) \leq \lim_{n\rightarrow \infty} E(\omega_n) -\frac{1}{2}\liminf _{n\rightarrow \infty}\int_\Pi \omega_n^2 dx\leq \mathcal{E}(\omega)\leq S_\mu^s.
\end{equation}
Thus, we must have $\mathcal{E}(\omega)= S^s_\mu$, which means that $\omega\in \Sigma^s_\mu$ is a maximizer. Moreover,
We can also conclude from \eqref{4-11} and \eqref{4-12} that $$\lim_{n\rightarrow \infty} \int_\Pi \omega_n^2 dx=\int_\Pi \omega^2 dx,$$ and hence
$$\omega_n\rightarrow \omega\quad\text{strongly in}\,\, L^2(\Pi).$$
Finally, we estimate
\begin{align*}
	\int_\Pi x_2|\omega_n-\omega|dx&= \int_{\Pi\cap B_R(0)} x_2|\omega_n-\omega|dx+\int_{\Pi\setminus B_R(0)} x_2|\omega_n-\omega|dx\\
	&\leq CR^2\|\omega_n-\omega\|_2+\int_{\Pi\setminus B_R(0)} x_2(\omega_n+\omega)dx\\
	&\leq CR^2\|\omega_n-\omega\|_2+\mu_n-\mu+2\epsilon.
\end{align*}
Letting $n\rightarrow \infty$ and then $\varepsilon\rightarrow 0$, we obtain
$$x_2\omega_n\rightarrow x_2\omega\quad\text{strongly in}\,\, L^1(\Pi).$$
The proof of Theorem \ref{cp} is thus complete.
\end{proof}

\section{Orbital Stability}
In this section, we study orbital stability of the solutions obtained in Theorem \ref{Ee}. Using Theorem \ref{cp}, we can obtain the following orbital stability for the set of maximizers.
\begin{theorem}\label{os}
Let $0<s<1$ and $0<\mu\leq \mu_0$. Then for arbitrary $M>0$ and $\varepsilon>0$, there exists $\delta>0$ such that for non-negative function $\xi_0\in L^1(\Pi)\cap L^{p_s}(\Pi)$ with $\|\xi_0\|_1\leq 1$, $\|\xi_0\|_{p_s}\le M$ and
	\begin{equation}\label{5-1}
		\inf_{\omega\in \Sigma^s_{\mu}} \Big{\{}\|\xi_0-\omega\|_2+\|x_2(\xi_0-\omega)\|_1\Big{\}}\leq \delta,
	\end{equation}
	if there exists a $L^{p_s}$-regular solution $\xi(t)$, $t\in [0, T)$ for some $0<T\leq \infty$  with the initial data $\overline{\xi_0}$, then
	\begin{equation}\label{5-2}
		\inf_{\omega\in \Sigma^s_{\mu}} \Big{\{}\|\xi(t)-\omega\|_2+\|x_2(\xi(t)-\omega)\|_1\Big{\}}\leq \varepsilon, \quad \forall\, t\in[0,T)
	\end{equation}
\end{theorem}
\begin{proof}
We argue by contradiction. Suppose on the contrary that there exists $\varepsilon_0>0$, non-negative functions $\xi_{0,n}\in L^1(\Pi)\cap L^{p_s}(\Pi)$ with $\|\xi_{0,n}\|_1\leq 1$, $\|\xi_{0,n}\|_{p_s}\le M_1$ for some $M_1>0$ independent of $\xi_0$, and $t_n\geq 0$ such that
	\begin{equation*}
		\inf_{\omega\in \Sigma^s_{\mu}} \Big{\{}\|\xi_{0,n}-\omega\|_2+\|x_2(\xi_{0,n}-\omega)\|_1\Big{\}}\leq \frac{1}{n},
	\end{equation*}
and
	\begin{equation}\label{5-3}
		\inf_{\omega\in \Sigma^s_{\mu}} \Big{\{}\|\xi_{n}(t_n)-\omega\|_2+\|x_2(\xi_{n}(t_n)-\omega)\|_1\Big{\}}\geq \varepsilon_0,
	\end{equation}
	where $\xi_n(t)$ is a $L^{p_s}$-regularity solution $\xi_n(t)$, $t\in [0, T_n)$ for some $T_n>0$ with the initial data $\xi_{0,n}$ and $0\leq t_n< T_n$.
	There exists $\omega_n\in \Sigma_\mu^s$ such that
	$$\|\xi_{0,n}-\omega_n\|_2+\|x_2(\xi_{0,n}-\omega_n)\|_1\leq \frac{2}{n}, \quad \forall\, n\geq 1.$$
	We infer from Lemma \ref{lm15} that $\omega_n$ equals to $\omega_\mu^s$ after a $x_1$-translation. Thus, we have $\|\xi_{0,n}-\omega_n\|_{p_s}<M_1+\|\omega_\mu^s\|_{p_s}$. By H\"older's inequality, we have
	\begin{align*}
		|E(\xi_{0,n})-E(\omega_n)|&=\left|\int_\Pi\int_\Pi (\xi_{0,n}-\omega_n)G_\Pi(x,y)(\xi_{0,n}+\omega_n)dxdy\right|\\
		&\leq \iint_{G_\Pi(x,y)>Rx_2y_2} |\xi_{0,n}-\omega_n|(x)G_\Pi(x,y)(\xi_{0,n}+\omega_n)(y)dxdy\\
		&\ \ \ \ \quad +\iint_{G_\Pi(x,y)\leq Rx_2y_2} |\xi_{0,n}-\omega_n|G_\Pi(x,y)(\xi_{0,n}+\omega_n)dxdy\\
		&\leq \iint_{|x-y|<CR^{-\frac{1}{4-2s}}} |\xi_{0,n}-\omega_n|(x)G_\Pi(x,y)(\xi_{0,n}+\omega_n)(y)dxdy\\
		&\ \ \ \ \quad+R\int_\Pi\int_\Pi |\xi_{0,n}-\omega_n|(x)x_2y_2(\xi_{0,n}+\omega_n)(y)dxdy\\
		&\leq CR^{-\frac{s-{1}/{p_s}}{2-s}}+CR||x_2(\xi_{0,n}-\omega_n)||_1.
	\end{align*}
	Hence, we deduce
	$$|\mathcal{E}(\xi_{0,n})-S^s_\mu|=|\mathcal{E}(\xi_{0,n})-\mathcal{E}(\omega_n)|\leq \|\xi_{0,n}-\omega_n\|_2+CR\|x_2(\xi_{0,n}-\omega_n)\|_1+CR^{-\frac{s-{1}/{p_s}}{2-s}},$$
	which implies $$\mathcal{E}(\xi_{0,n})\rightarrow S^s_\mu,$$
	by first letting $n\rightarrow \infty$ and then $R\rightarrow \infty$.
	By the conservation laws, functions $\xi^n:=\xi_{n}(t_n)$ satisfy
	\begin{equation*}
		\begin{cases}
			\xi^n\geq 0,\ \xi^n\in L^1(\Pi)\cap L^p(\Pi),\ \int_\Pi \xi^n dx\leq 1,\ \|\xi^n\|_{p_s}\le M_1,\\
			\mu_n=\int_\Pi x_2\xi^n dx\rightarrow \mu, \quad  \text{as} \  n\rightarrow \infty,\\
			\mathcal{E}(\xi^n)\rightarrow S^s_{\mu},\quad  \text{as} \ n\rightarrow \infty,
		\end{cases}
	\end{equation*}
	Therefore, we infer from Theorem \ref{cp} that there exists $\omega\in \Sigma^s_{\mu}$, a subsequence $\{\xi^{n_k}\}_{k=1}^\infty$ and a sequence of real number $\{c_k\}_{k=1}^\infty$ such that
	\begin{equation*}
		\|\xi^{n_k}(\cdot+c_k\mathbf{e}_1)-\omega\|_2+\|x_2(\xi^{n_k}(\cdot+c_k\mathbf{e}_1)-\omega)\|_1\rightarrow 0,
	\end{equation*}
	as $k\rightarrow \infty$. This is contrary to \eqref{5-3}, and hence the proof of Theorem \ref{os} is completed.	
\end{proof}

\phantom{s}
 \thispagestyle{empty}

\end{document}